\documentclass[11pt]{amsart}
\usepackage{graphicx}
\usepackage[active]{srcltx}
 \makeatletter
\renewcommand*\subjclass[2][2000]{%
  \def\@subjclass{#2}%
  \@ifundefined{subjclassname@#1}{%
    \ClassWarning{\@classname}{Unknown edition (#1) of Mathematics
      Subject Classification; using '1991'.}%
  }{%
    \@xp\let\@xp\subjclassname\csname subjclassname@#1\endcsname
  }%
}
 \makeatother
\usepackage{enumerate,url,amssymb,  mathrsfs,yhmath}

\newtheorem{theorem}{Theorem}[section]
\newtheorem{lemma}[theorem]{Lemma}
\newtheorem*{lemma*}{Lemma}

\newtheorem{corollary}[theorem]{Corollary}

\theoremstyle{definition}
\newtheorem{definition}[theorem]{Definition}

\theoremstyle{remark}
\newtheorem{remark}[theorem]{Remark}

\numberwithin{equation}{section}

\newcommand{\abs}[1]{\lvert#1\rvert}

\def\XXint#1#2#3{{\setbox0=\hbox{$#1{#2#3}{\int}$}
\vcenter{\hbox{$#2#3$}}\kern-.5\wd0}}

\def\le{\leqslant}
\def\ge{\geqslant}
\begin{document}

\title[Quasiconformal harmonic mappings between surfaces]{Quasiconformal harmonic mappings between $\,\mathscr C^{1,\mu}$ Euclidean
surfaces} \subjclass{Primary 35J05; Secondary 47G10}


\author{David Kalaj}
\address{University of Montenegro, Faculty of Natural Sciences and
Mathematics, Cetinjski put b.b. 81000 Podgorica, Montenegro}
\email{davidk@t-com.me}

\begin{abstract} The conformal deformations are contained in two classes of mappings:
quasiconformal and harmonic mappings. In this paper we consider the
intersection of these classes. We show that, every $K$
quasiconformal harmonic mapping between $\,\mathscr C^{1,\mu}$
($0<\alpha\le 1$) surfaces with $\,\mathscr C^{1,\mu}$ boundary is a
Lipschitz mapping. This extends some recent results of several
authors where the same problem has been considered for plane
domains. As an application it is given an explicit Lipschitz
 constant of normalized isothermal coordinates of a disk-type minimal surface in terms of
 boundary curve only. It seems that this kind of estimates are new for conformal
 mappings of the unit disk onto a Jordan domain as well.
\end{abstract} \maketitle


\section{Introduction and statement of the main result}

Let $u:D\to \mathbf R^n$ be a continuous mapping defined in a domain
$D$ of the complex plane $\,\mathbf C = \{ z= x+iy\,,\;\; x , y
\,\in \mathbf R\,\}$, given by
$$u(z)=(u^1(z),u^2(z),\dots, u^n(z)),$$ having partial
derivatives in a point $z:=x+i y\in D$. The formal derivative
(Jacobian matrix) of $u$ in $z$ is defined by $$\nabla u(z)=\left(
                                     \begin{array}{cc}
                                       u^1_x &  u^1_y \\
                                       \vdots &  \vdots \\
                                       u^n_x &  u^n_y \\
                                     \end{array}
                                   \right).$$
 Then the Jacobian determinant
(area magnification factor) of $u$ and the Hilbert-Schmidt norm of
$\nabla u(z)$ are defined by
$$J_u(z)=\left(\det [\nabla u(z)^T\cdot \nabla
u(z)]\right)^{1/2}=\sqrt{|u_x|^2|u_y|^2-\left<u_x,u_y\right>^2}$$
where by $\left<\cdot,\cdot \right>$ and $|\cdot|$ are denoted the
standard inner product and standard Euclidean norm {in the space}
$\mathbf R^n$, and
$$\|\nabla u(z)\|= \left(\frac 12\mathbf{Tr}\,\nabla u(z)^T\cdot \nabla
u(z)\right)^{1/2}=\sqrt{\frac 12(|u_x|^2+|u_y|^2)}.$$  We also will
use the following operator norm of $\nabla u$: $$|\nabla
u(z)|=\max\{|\nabla u(z)h|:|h|=1, h\in \mathbf C\},$$ and notice
that, in the case of conformal mappings, the following two norms
coincide.

\subsection{Parametric Surfaces}

We define an oriented parametric surface $\, \mathcal M\,$ in
$\,\mathbf R^n\,$ to be an equivalence class of mappings $\,u =
(u^1, \dots , u^n) : D \rightarrow \mathbf R^n\,$ of some domain
$\,D \subset \mathbf C\,$ into $\,\mathbf R^n\,$, where the
coordinate functions $\, u^k =u^k(x,y)$, $k=1,\dots, n$ are of class
at least $\,\mathscr C^1(D)\,$. Two such mappings $u: D \rightarrow
\mathbf R^n\,$ and $\,\tilde{u} : \tilde{D} \rightarrow \mathbf
R^n\,$, referred to as parametrizations of the surface, are said to
be equivalent if there is a $\,\mathscr C^1$-diffeomorphism $\,\phi
:\tilde{D}\overset{\textnormal{\tiny{onto}}}{\longrightarrow} D \,$
of positive Jacobian determinant  such that $\; \tilde{u} = u\circ
\phi\,$.  Let us call such $\,\phi\,$ a \textit{change of variables,
or reparametrization} of the surface. Furthermore, we assume that
the branch (critical) points of  $\, \mathcal M\,$  are isolated.
These are the points  $\,(x,y) \in D\,$  at which the tangent
vectors $\,u_x = \frac{\partial u}{\partial x},\; u_y =
\frac{\partial u }{\partial y}\,$ are linearly dependent.
Equivalently, at the critical points the \textit{Jacobian matrix}
$\nabla u(z)$ has rank at most 1. It has full rank  2 at the
\textit{regular points}.
A surface with no critical points is called an \textit{immersion} or
a \textit{regular surface}. If $u:\mathbf U\to \mathcal M$, then the
surface $\mathcal M=u(\overline{\mathbf U})$ is called a
\textit{disk-type} surface with boundary. If $u$ has continuous
extension to the boundary, then it defines the boundary
$\partial\mathcal M:=u(\mathbf T)$ of the surface $\mathcal M$. If
$u|_{\mathbf T}: \mathbf T\to \partial\mathcal M$ is a
diffeomorphism, then the boundary is immersed.

The area of the surface equals
\begin{equation}
 \left| u(D)\,\right|  \, =\, \iint_D  J_u\; \textrm{d}x\,
 \textrm{d}y.
\end{equation}
\subsection{Harmonic mappings}
A mapping $u=(u^1,\dots,u^n):D \to \mathbf R^n$ is called
\emph{harmonic} in a region $D\subset \mathbf C$ if for $k=1,\dots,
n$, $u^k$ is real-valued harmonic functions in $D$; that is $u^k$ is
twice differentiable and satisfies the Laplace equation $$\Delta u^k
:=u^k_{xx} + u^k_{yy}= 0.$$ Let
$$P(r,t)=\frac{1-r^2}{2\pi (1-2r\cos t+r^2)},\; 0\le r<1,\; 0\le t\le 2\pi$$ denote the Poisson
kernel. Then every bounded harmonic mapping $u:\mathbf U \to \mathbf
R^n$, $n\ge 1$, defined on the unit disc $\mathbf U:=\{z:|z|<1\}$,
has the following representation
\begin{equation}\label{e:POISSON}
u(z)=P[F](z)=\int_0^{2\pi}P(r,t-\varphi)F(e^{it})dt,
\end{equation}
where $z=re^{i\varphi}$ and $F$ is a bounded integrable function
defined on the unit circle $\mathbf T: =\{z:|z|=1\}$.
%
%
%
If $u$ is a harmonic mapping, then the surface $\mathcal
M:=u(\mathbf U)$ is called a \textit{harmonic surface}. If $\mathcal
M$ is a plane surface, then according to Lewy's theorem (\cite{l}),
$J_u>0$ providing that $u$ is a homeomorphism. On the other hand, by
a result of Berg (\cite{berg}) and Lichtenstein theorem, if
$\mathcal M$ is $C^{1,\alpha}$ regular, then $J_u>0$ on $D$,
providing that $u$ is a homeomorphism.

\subsection{Quasiconformal mappings}
We say that a mapping $u:D\to \mathbf R^n$ is ACL (absolutely
continuous on lines) in  the region $D\subset \mathbf C$, if for
every closed rectangle $R\subset D$ with sides parallel to the $x$
and $y$-axes, $u$ is absolutely continuous on a.e. horizontal and
a.e. vertical line in $R$. Such a function has of course, partial
derivatives $u_x$, $u_y$ a.e. in $D$.

A homeomorphism  $u:D\to u(D)\subset \mathbf R^n$ is said to be
$K$-\textit{quasiconformal} ($K$-q.c), $K\ge 1$, if $u$ is ACL and
if
\begin{equation}\label{defin} \|\nabla u(z)\|^2\le\frac 12\left(K+\frac
1K\right)J_u(z),\ \ |z|<1.
\end{equation}
We will say that a q.c. mapping $f:\mathbf U\to \mathcal M$, of the
unit disk onto a disk-type surface with rectifiable boundary
$\gamma=\partial \mathcal M$ is \textit{normalized} if $f(1)=w_0$,
$f(e^{2\pi i/3}) = w_1$ and $f(e^{4\pi i/3 }) = w_2$, where
$\wideparen{w_0w_1}$, $\wideparen{w_1w_2}$ and $\wideparen{w_2w_0}$
are arcs of $\gamma=\partial \mathcal M$ having the same length
$|\gamma|/3$, where $|\gamma|$ is the length of $\gamma$. (Such
mappings have continuous extension to the boundary.)

Assume in addition that $u$ is harmonic. Then $$u(z)=(\mathrm{Re}\,
a_1(z),\dots,\mathrm{Re}\,a_n(z)),$$ for some analytic functions
$a_k(z)$, $z\in\mathbf U$, $k=1,\dots,n$. Since
$$\|\nabla u(z)\|^2=\frac 12\sum_{k=1}^n |a'_k(z)|^2,$$ it follows from
\eqref{defin} that $J_u(z)>0$ except for some isolated points in
$\mathbf U$ which are branch points of the \textit{harmonic
quasiconformal surface} $\mathcal M:=u(D)$.

\subsection{Isothermal (conformal) parameters}
In what follows, we will concern ourselves mostly conformal
parametrizations (1 - quasiconformal mappings) $\,u= (u^1,\dots,
u^n) : \Omega \rightarrow \mathbf R^n\,$.
 This simply means that the coordinate functions,  called \textit{isothermal parameters}, will satisfy the conformality relations:
If $K = 1$ then \eqref{defin} is equivalent to the system of the
equations
 \begin{equation}\nonumber
 \;\begin{split}
  \;\begin{cases} \sum_{k=1}^nu^k_xu^k_y =  0  \,, \quad\quad\quad\quad\; \;\;\;\,(\,u_x \;\textrm{and} \; u_y \; \textrm{are orthogonal in} \;\;
  \mathbf R^n\,)\\
 \sum_{k=1}^n(u^k_x)^2=\sum_{k=1}^n(u^k_y)^2 \;\,   \,\;\quad\quad\; \; (\,u_x \;\textrm {and} \;u_y\;\;\textrm{have equal length\;).}
 \end{cases}\end{split}
 \end{equation}
 Equivalently, it means that:
\begin{equation}\label{surfeq1}
 J_u \;= \; |u_x|^2\;  = |u_y|^2. \;\;
 \end{equation}
Thus $\,u\,$ is an  immersion if $\,\nabla u \,\neq 0$ at every
point. We refer to~\cite{DHKW} for an excellent historical account
of existence of isothermal coordinates. If $u$ is harmonic and
satisfies the system \eqref{surfeq1}, then $\mathcal M$ is a minimal
surface. A minimizing surface is a minimal surface spanning a given
curve and having the least area.

\begin{remark}\label{rema} A celebrated theorem by Lichtenstein
states that each regular surface $\mathcal M:=u(\Omega)\subset
\mathbf R^n$ of class $\,\mathscr C^{1,\mu}$, $0 < \mu < 1$, can be
mapped conformally onto a planar domain $\Omega$.
%
In particular, if $\mathcal M$ is $\,\mathscr C^{1,\mu}$ disk-type
surface with $\,\mathscr C^{1,\mu}$ immersed boundary, then there
exists a conformal mapping $\tau$ of the unit disk onto $\mathcal M$
such that\begin{equation}\label{cm}C_{\mathcal M}:=\sup_{z\neq
w}\frac{|\tau'(z)-\tau'(w)|}{|z-w|^\alpha}<\infty\end{equation} and
for some $c_{\mathcal M}>0$
\begin{equation}\label{cma}c_{\mathcal M}\le |\tau'(z)|\le \frac{1}{c_{\mathcal M}}, \ z\in\mathbf
U,\end{equation} and \begin{equation}\label{cmal}\frac 1{c_{\mathcal
M}}\le \frac{|\tau(z)-\tau(w)|}{|z-w|}\le c_{\mathcal
M}.\end{equation}
\end{remark}

\subsection{Some background and statement of the main result} The main question addressed here is under which conditions,
a given harmonic diffeomorphism (or a given quasiconformal mapping)
between smooth surfaces with smooth boundary  is globally Lipschitz
continuous. In the best known situation, where the domain and image
domain is the unit disk, neither harmonic diffeomorphisms, neither
quasiconformal mappings are Lipschitz. Under some additional
conditions on the dilatation, a quasiconformal self-mapping of the
unit disk is Lipschitz at the boundary (see \cite{anhi}). Conformal
mappings (minimal surfaces and minimizing surfaces) are a subclass
of both classes: harmonic and quasiconformal mappings. Concerning
the regularity and bi-Lipschitz character of minimal surfaces and
minimizing surfaces we refer to \cite{nit},\cite{hh}, \cite{les} and
\cite{bw}. In this paper, we will consider a bit more general
situation. We will consider harmonic quasiconformal mappings between
surfaces and investigate their Lipschitz character up to the
boundary. This topic of research has its origin in the classical
paper of Martio \cite{om}. See also \cite{hes} and \cite{Pd} for
related results. In some recent results, see
\cite{mathz}-\cite{MMM}, \cite{mv2}, \cite{kojic},
\cite{ps}-\cite{MP} is established the Lipschitz and bi-Lipschitz
character of quasiconformal harmonic mappings between smooth Jordan
domains. In \cite{wan}, \cite{tamwan} and \cite{tam1} a similar
problem for quasiconformal harmonic mappings with respect to the
hyperbolic metric is treated. The class of quasiconformal harmonic
mappings has been showed interesting, due to a recent discovery,
that a q.c. harmonic mapping makes smaller distortion of moduli of
annulus than a quasiconformal mapping \cite{nits}.

Recently in \cite{km}, it is shown that, if $u$ is a quasiconformal
harmonic mapping of the unit disk onto a smooth $\,\mathscr
C^{2,\mu}$ surface with $\,\mathscr C^{2,\mu}$ boundary, then $u$ is
Lipschitz.

In this paper we replace the condition $\,\mathscr C^{2,\mu}$ by
$\,\mathscr C^{1,\mu}$ and find explicit Lipschitz constant for a
normalized q.c. harmonic mapping. The celebrated Kellogg's theorem
(see e.g. \cite{G}), implies that a conformal mapping of the unit
disk of onto a Jordan domain with $\,\mathscr C^{1,\alpha}$ boundary
$\gamma$ is Lispchitz continuous. However, until now, it is not
known any explicit estimation of Lipschitz constant, depending on
$\gamma$.

Our main result can be stated as follows.

\begin{theorem}\label{thm}
Let $u: \mathbf U\to u(\mathbf U)\subset\mathbf R^n$ be a
$K-$quasiconformal harmonic mapping of the unit disk onto a
$\,\mathscr C^{1,\mu}$ surface $\mathcal M=u(\mathbf U)$ with
$\,\mathscr C^{1,\mu}$ boundary $\gamma$. Then $u$ is Lipschitz i.e.
there exists a constant $L$ such that
$$|\nabla u(z) |\le L,\ \ z\in\mathbf U$$ and
$$|u(z)-u(w)|\le KL|z-w|, \ z,w\in\mathbf U.$$
If $u$ is a normalized q.c. mapping, then $L$ depends only on $K$
and $\gamma$, and it not depends on $u$ neither on the surface
$\mathcal M$. It satisfies the inequality~\eqref{L} below.
\end{theorem}
The proof of Theorem~\ref{thm} is given in the third section. The
proof rely on several lemmas which are proved in the second section,
however the main ingredient of the proof is Lemma~\ref{newle} (which
can be considered as a Mori's theorem for q.c. mappings of the unit
disk onto a smooth surface).

The proofs are
 different form the proofs of related results in \cite{km} where is imposed $\,\mathscr C^{2,\mu}$ regularity of
 surface and some properties of the second derivative of conformal
 parametrization. Some of the tools for the proofs are conformal
parametrization of a surface, arc-length parametrization of its
boundary and isoperimetric inequality.
 We would like to point out Corollary~\ref{coco} where is given an
 application of Theorem~\ref{thm}: it is given an explicit Lipschitz
 constant of isothermal coordinates of a minimal surface in terms of
 boundary curve only.

Recall that the family of quasiconformal harmonic mappings contains
conformal mappings. In \cite{Lw} is given an example of a
$\,\mathscr C^1$ Jordan curve $\gamma$, such that a conformal
mapping of the unit disk $\mathbf U$ onto
$D=\mathrm{int}(\gamma)\subset \mathbf C$ is not Lipschitz. On the
other hand, in \cite{kalajpub} it is given an example of harmonic
quasiconformal mapping of the unit disk onto itself, that is not
smooth up to the boundary. This in turn implies that, our result is
the best possible in this context.

\section{Auxiliary results}
\begin{lemma}
Let $u$ be $K-$quasiconformal and $z=re^{it}$. Then
\begin{equation}\label{december1} \left|\frac{\partial u}{\partial
t}\right|^2\le r^2 K J_u(z).\end{equation}
\end{lemma}
\begin{proof} It is easily to obtain that, the condition \eqref{defin} is
equivalent to
\begin{equation}\label{1}|\nabla u(z)|\le Kl(\nabla u(z)),\end{equation}
where $$|\nabla u(z)|:=\max\{|\nabla u(z) h|:|h|=1\}$$ and
$$l(\nabla u(z)):=\min\{|\nabla u(z) h|:|h|=1\},$$ (see e.g.
\cite[Lemma~2.1]{matkal}). Let $h = (\alpha, \beta)\in \mathbf T$.
Then
$$|\nabla u(z) h|^2 = |u_x|^2\alpha^2 + 2\left<u_x,u_y\right> \alpha
\beta +|u_y|^2\beta^2.$$ This yield \begin{equation}\label{2}\max
\{|\nabla u(z) h|:|h|=1\} =
\sqrt{\frac{(|u_x|^2+|u_y|^2)(1+\sqrt{1-4\eta^2})}{2}}\end{equation}and
\begin{equation}\label{3}\min\{|\nabla u(z) h|:|h|=1\}
=\sqrt{\frac{(|u_x|^2+|u_y|^2)(1-\sqrt{1-4\eta^2})}{2}},\end{equation}
where
$$\eta = \frac{(|u_x|^2\cdot
|u_y|^2-\left<u_x,u_y\right>^2)^{1/2}}{|u_x|^2+|u_y|^2}.$$ Therefore
\begin{equation}\label{inter}J_u(z)=|\nabla u(z)|\cdot l(\nabla
u(z)).\end{equation} For  $ z = re^{it}$ we have
\begin{equation}\label{var}\frac{\partial u}{\partial t} =r
u_y \cos t - r u_x \sin t.\end{equation} Thus
\begin{equation}\label{igi} r l(\nabla u)\le \left|\frac{\partial
u}{\partial t}\right|\le r |\nabla u|.
\end{equation}
By \eqref{igi}, \eqref{inter} and \eqref{1} we obtain
\begin{equation}
\left|\frac{\partial u}{\partial t}\right|^2\le r^2 K J_u(z).
\end{equation}
\end{proof}
%
\begin{definition}
For a positive nondecreasing continuous function $\omega$,
$\omega(0)=0$ we will say that is Dini's continuous if it satisfies
the condition
\begin{equation}\label{dini}\int_{0}^{l} \frac{\omega(t)}{t}
dt<\infty.\end{equation} A smooth Jordan curve $\gamma$, is said to
be Dini's smooth if there exists a $\,\mathscr C^1$ diffeomorphism
$h:\mathbf T\to \gamma$, such that the modulus of continuity
$\omega$ of $h'$ is Dini's continuous. Observe that every smooth
$\,\mathscr C^{1,\mu}$ Jordan curve is Dini's smooth.
\end{definition}
\begin{lemma}
If $\omega$ is Dini continuous in $[0,l]$, then $\omega$ has Dini
continuous extension in $[0,L]$, ($L>l$), which will be also denoted
by $\omega$. Moreover for every constant $a$, $\omega(ax)$ is Dini
continuous. Next for every $0<y\le l$ there holds the following
formula:
\begin{equation}\label{goodd}\int_{0+}^y\frac{1}{x^2}\int_0^x\omega(a t) dt dx =
\int_{0+}^y\frac{\omega(ax)}{x}-\frac{\omega(ax)}{y}dx.\end{equation}
\end{lemma}
\begin{proof}
Taking the substitutions $u = \int_0^x\omega(a t) dt $ and $dv =
x^{-2} dx$, and using the fact that $$\lim_{\alpha \to
0}\frac{\int_0^{\alpha}\omega(a t) dt}{\alpha} = \omega(0) = 0$$ we
obtain:
\[\begin{split}\int_{0+}^y\frac{1}{x^2}\int_0^x\omega(a t) dt dx
&=\lim_{\alpha\to 0+}\int_{\alpha}^y\frac{1}{x^2}\int_0^x\omega(a t)
dt dx\\& =-\lim_{\alpha\to 0+}\left.\frac{\int_0^x\omega(a t)
dt}{x}\right|_{\alpha}^{y} +\lim_{\alpha\to
0+}\int_{\alpha}^{y}\frac{\omega(a
x)}{x}dx\\&=\int_{0+}^y\frac{\omega(ax)}{x}-\frac{\omega(ax)}{y}dx.\end{split}\]
\end{proof}
 Let $\gamma\in \,\mathscr C^{1}$, and $h:\mathbf T\to \gamma$ be a smooth function.
 We will sometimes write $h(t)$ and  $h'(t)$ instead of $h(e^{it})$ and $\frac{d}{dt} h(e^{it})$.
Consider the following function
\begin{equation}\label{kerk}\mathcal{K}_h(e^{is},e^{it})=\sqrt{|h(t)-h(s)|^2|h'(s)|^2-\left<h(t)-h(s),
h'(s)\right>^2}.\end{equation} Wee need the following lemma.
\begin{lemma}\label{11}
If $h:\mathbf T\to \gamma$ is Dini's smooth function of the unit
circle onto a Dini's smooth Jordan curve $\gamma$, and $\omega$ is
modulus of continuity of $h'$, then
\begin{equation}\label{oh}|\mathcal{K}_h(e^{is},e^{it})|\le
\frac{|h(e^{is})-h(e^{it})|}{|e^{is}-e^{it}|}\int_0^{\pi
|e^{is}-e^{it}|}\omega(\tau)d\tau.\end{equation} In particular, if
$h\in \,\mathscr C^{1,\mu}$, then
\begin{equation}\label{buki}|\mathcal{K}_h(e^{is},e^{it})|\le
c_h|h(e^{is})-h(e^{it})|{|e^{is}-e^{it}|^\mu},\end{equation} where
$$c_h = \frac{1}{1+\mu}\sup_{x\neq y} \frac{|h'(x)-h'(y)|}{|x-y|^\mu}.$$
Moreover, if $\varphi(e^{it})=e^{if(t)}$ is a smooth mapping of
$\mathbf T$ onto itself, then
\begin{equation}\label{top}|\mathcal{K}_{h\circ
\varphi}(e^{is},e^{it})|=|f'(s)||\mathcal{K}_h(e^{if(s)},e^{if(t)})|.\end{equation}
\end{lemma}

\begin{proof}
Set $X =h(t)-h(s)$, $Y = h'(s)$ and $\alpha\in \mathrm R$. Then
\[\begin{split}|&X|^2|Y+\alpha X|^2-\left<X,Y+\alpha X\right>^2
\\&=|X|^2(|Y|^2+2\alpha\left<X,Y \right>+\alpha^2|X|^2)-\left<X,Y\right>^2-2\alpha\left<X,Y\right>|X|^2-\alpha^2|X|^4\\&=
|X|^2|Y|^2-\left<X,Y\right>^2.
\end{split}\]

 Therefore we obtain

\[\begin{split}\mathcal{K}_h(s,t)&=\sqrt{|X|^2|Y|^2-\left<X,Y\right>^2}\\&=
 \sqrt{|X|^2|Y+\alpha X|^2-\left<X,Y+\alpha X\right>^2}\\&\le\sqrt{|X|^2|Y+\alpha
 X|^2}=|X||Y+\alpha
 X|
.\end{split}\] Take now $$\alpha = \frac{1}{s-t}.$$ Since
$$Y+\alpha
 X=h'(s)-\dfrac{h(t)-h(s)}{t-s}=\int_s^t\frac{h'(s)-h'(\tau)}{t-s}d\tau,$$
we obtain
\[\begin{split}\left|h'(s)-\dfrac{h(t)-h(s)}{t-s}\right|&\le\int_s^t
\frac{|h'(s)-h'(\tau)|}{t-s}d\tau\\ &\le \int_s^t
\frac{\omega(\tau-s)}{t-s}d\tau
\\&=\frac{1}{t-s}{\int_0^{t-s}\omega(\tau)d\tau}.
\end{split}\] As $|X|=|{h(t)-h(s)}|,$ we obtain

\begin{equation}\label{ohoh}|\mathcal{K}_h(e^{is},e^{it})|\le
\frac{|h(s)-h(t)|}{|s-t|}\int_0^{|s-t|}\omega(\tau)d\tau.\end{equation}
Since $\mathcal{K}_h(e^{i(s\pm 2\pi)}, e^{i(t\pm 2\pi)}) =
\mathcal{K}_h(e^{is}, e^{it})$, according to \eqref{ohoh} and
$$\frac{1}{\pi}\min\{|s-t|,2\pi-|s-t|\}\le |e^{is}-e^{it}|\le \min\{|s-t|,2\pi-|s-t|\}$$ we obtain
\eqref{oh}. The inequality \eqref{buki} follows from \eqref{oh}. On
the other hand
$$\frac{d}{dt}h(e^{if(t)})=h'(f(t))f'(t)$$
and this yields \eqref{top}.
\end{proof}


\begin{definition}\label{perdef}
Let $0<\Upsilon\le {\pi}$. A smooth surface $\mathcal M$ is said to
be $\Upsilon-${\it isoperimetric } if every rectifiable Jordan curve
$\gamma\subset \mathcal M$ with the length $L$ spans a subsurface
$M_\gamma\subset \mathcal M$ with the area $A$ satisfying the
inequality
\begin{equation}\label{isopercoef}
\frac{A}{L^2}\le \frac{1}{4\Upsilon}.
\end{equation}
\end{definition}
\subsection{Examples}
\begin{itemize}
    \item{Carleman, \cite[Theorem~3.5, p.~129--232]{cou}}.  Every minimal surface (in particular every complex Jordan domain) is isoperimetric with $\Upsilon=
    {\pi}$.
    \item{Courant, \cite[Theorem~3.7.~(The proof)]{cou}}. Every harmonic surface is isoperimetric with $\Upsilon=
    {1}$.
    \item{Huber, \cite{hub}}. If the integral mean of Gauss curvature $K$ satisfies
    the inequality $$\int_{\mathcal M} \max\{K,0\}dM < 2\pi,$$ then $$\Upsilon =
    {\pi -\frac 12\int_{\mathcal M} \max\{K,0\} dM}.$$
     \item According to the previous item, every $\,\mathscr C^2$ surface is locally isoperimetric.
\item{Heinz, \cite{hein}}.  Every constant mean curvature surface $x=u(z)$, $z\in \overline{G}$ with mean curvature $H$ satisfying
     $h:=|H|\max_{z\in\overline G}|x-c|<1$, $c\in \Bbb
     R^3$, is a isoperimetric  surface with
    $$\Upsilon ={\pi}\frac{1-h}{1+h}.$$
    \item{} Li\&Tam \cite{lifa}. Let $M$ be a complete noncompact surface with finite total
curvature. Suppose that all the ends of $M$ have quadratic area
growth. Then $M$ is an isoperimetric surface for some constant
$\Upsilon$.

\item{Kalaj\&Mateljevi\'c}, \cite{matkal}. Every quasiconformal
harmonic surface with rectifiable boundary is a isoperimetric
surface with $\Upsilon = \max\{\frac{2\pi}{1+K^2},1\}$.
\end{itemize}

\begin{lemma}\label{22}
If $u=P[F]$ is a harmonic mapping, such that $F$ is a Lipschitz weak
homeomorphism from the unit circle onto a Dini's smooth Jordan curve
$\gamma=h(\mathbf T)$ spanning a surface $u(\mathbf U)=\mathcal
M\subset\mathbf R^n$, then for almost every $\tau\in [0,2\pi]$ there
exists
$$J_u(e^{i\tau}) :=\lim_{r\to 1} J_u(re^{i\tau})$$ and there holds the
inequality

\begin{equation}\label{jacfor}\begin{split}J_u(e^{i\tau})&\le f'(\tau)\int_0^{2\pi}
\frac{\mathcal{K}_h(e^{if(\tau)},e^{if(t)})}{4\pi\sin^2\frac{t-\tau}{2}}dt<\infty,\end{split}\end{equation}
where $$F(e^{it}) = h(e^{if(t)}),$$ and $\mathcal{K}_h$ is defined
in \eqref{kerk}.
\end{lemma}

\begin{proof}
We will establish the existence of partial derivatives at the
boundary of $u_r$ and $u_{\varphi}$. Let $$u(z)=(u_1(z),\dots,
u_n(z))=P[(F_1,\dots, F_n)].$$ Observe first that, for $i=1,\dots,
n$, there exists an analytic function $h_i$ such that
  $$
u_i(z)=h_i(z)+\overline {h_i(z)}= \frac
1{2\pi}\int_0^{2\pi}\frac{e^{it}}{e^{it}-z}F_i(e^{it})d t+ \frac
1{2\pi}\overline {\int_0^{2\pi}\frac{z}{e^{it}-z} {F_i(e^{it})}dt}.
$$ It follows that
\begin{equation}\label{spliteq}
\begin{split}
zh_i'(z)&=\frac{z}{2\pi}\int_0^{2\pi}\frac{e^{it}}{(e^{it}-z)^2}F_i(e^{it})dt
=\frac{1}{2\pi i}\int_0^{2\pi}\frac
{z}{e^{it}-z}F_i'(e^{it})dt\\&=\frac 1{2 \pi i}\int_0^{2 \pi }\frac
{e^{i t}}{e^{it}-z}F_i'(e^{it})d t.
\end{split}
\end{equation} Now in view of the fact that $ \mathrm{ Re}\frac {e^{i \tau}}{e^{i\tau}-z}>0$,
according to  (\cite[Theorem~2.2]{Pd}, see also \cite{Rl}), we get
that there exist radial boundary values of the function $zh_i'(z)$
almost everywhere.
\\
It follows that, there exists $$ J_u(e^{i\tau}):=\lim_{r \to 1}
J_u(re^{i\tau})$$ for almost every $\tau \in [0,2\pi]$.
\\
Now by using the fact that $F' \in L^1(\mathbf T)$ we get
\begin{equation}\label{e:REEQ}
\lim_{r\to 1}u_{\tau}(re^{i\tau})=u_\tau(e^{i\tau})
\end{equation}
for almost every $e^{i\tau}\in \mathbf T$.
\\
Let $e_3(r,\varphi),e_4(r,\varphi), \dots, e_{n}(r,\varphi)$ be an
orthonormal system of vectors orthogonal to $u_r$ and $u_\varphi$ in
the Euclidean space $\mathbf R^n$.

Define the vector $$\upsilon(r,\varphi)= u_r(re^{i\varphi})\times
u_\varphi(re^{i\varphi})\times e_3(r,\varphi)\times\dots\times
e_{n-1}(r,\varphi).$$ Then $\upsilon(r,\varphi)$ is collinear with
$e_{n}(r,\varphi)$ and its norm is given by
$$|\upsilon(r,\varphi)|=J_u(re^{i\varphi}).$$

Now for constant vectors $\chi_2,\dots,\chi_{n-2}$ and for the
vector functions $\chi(t)$ and $\psi(t)$  there hold
\begin{equation}\label{eqvec}\int_a^b \chi(t)dt \times
\chi_2\times\dots\times \chi_{n-2}=\int_a^b \chi(t) \times
\chi_2\times\dots\times \chi_{n-2}dt,\end{equation}

and \begin{equation}\label{vec}|\int_a^b \psi(t) dt|\le
\int_a^b|\psi(t)|dt.\end{equation}

By using \eqref{eqvec}, \eqref{vec} and the fact that $u$ is a
harmonic mapping we obtain:
\begin{equation}\label{last}
\begin{split}
&\lim_{r \to 1} J_u(re^{i\tau})=\lim_{r \to 1}
\left|u_r(re^{i\tau})\times u_\tau(re^{i\tau})\times
e_3(r,\tau)\times\dots\times e_{n-1}(r,\tau)\right|
\\ & =\lim_{r \to 1} \left|\frac {u(r e^{i\tau})-u(e^{i\tau})}{1-r}
\times u_\tau(e^{i\tau})\times e_3(r,\tau)\times\dots\times
e_{n-1}(r,\tau)\right|\\ &= \lim_{r\to 1}\left|\int_{-\pi}^{\pi}
\left( {u( e^{it})-u(e^{i\tau})}\right) \times
u_\tau(e^{i\tau})\times\dots\times e_{n-1}(r,\tau)\frac
{P(r,\tau-t)}{1-r}dt\right|\\&\le \lim_{r\to 1}\int_{-\pi}^{\pi}
\mathcal{K}_F(e^{i(t+\tau)},e^{i\tau})\frac {P(r,t)}{1-r}dt,
\end{split}
\end{equation}
where $P(r,t)$ is the Poisson kernel and
\begin{equation*}\begin{split} \mathcal{K}_F(e^{it},e^{i\tau})&=\sqrt{| {u(
e^{it})-u(e^{i\tau})}|^2 |u_\tau(e^{i\tau})|^2-\left<u(
e^{it})-u(e^{i\tau}),u_\tau(e^{i\tau})\right>^2}\\&=\sqrt{| {F(
e^{it})-F(e^{i\tau})}|^2 |F'(e^{i\tau})|^2-\left<F(
e^{it})-F(e^{i\tau}),F'(e^{i\tau})\right>^2}\\&=f'(\tau)\sqrt{|h(f(t))-h(f(\tau))|^2-\left<h(f(t))-h(f(\tau)),
h'(f(\tau))\right>^2},\end{split}\end{equation*} i.e.
\begin{equation}\label{K}\mathcal{K}_F(e^{it},e^{i\tau})=
f'(\tau)\mathcal{K}_h(f(t),f(\tau)).\end{equation} To continue,
observe first that
$$\frac{P(r,t)}{1-r}=\frac {1+r}{2\pi(1+r^2-2r\cos t)}
\leq \frac 1{\pi((1-r)^2+4r\sin^2t/2)}\le \frac{\pi}{4rt^2}$$ for
$0<r<1$ and $t\in [-\pi,\pi]$.

On the other hand by \eqref{oh}, \eqref{top} and \eqref{K}, for
$$\sigma = \pi |e^{if(t+\tau)}-e^{if(\tau)}|$$ and $$|g|_\infty:=\mathrm{ess}\,\sup_{t}|g(t)|,$$ we obtain
$$\abs{\mathcal{K}_F(e^{i(t+\tau)},e^{i\tau})}\le |h'|_\infty|f'|_{\infty}\int_0^{\sigma}\omega(x)dx.$$
Therefore
\begin{equation}\label{eee}\begin{split}\left|\mathcal{K}_F(e^{i(t+\tau)},e^{i\tau})\frac
{P(r,t)}{1-r}\right|&\le \frac{|h'|_\infty|f'|_{\infty}\pi}{4r
t^2}\int_0^{\sigma}\omega(u)du\\&\le\frac{\sigma}{t}\frac{|h'|_\infty|f'|_{\infty}\pi}{4r
t^2}\int_0^{t}\omega\left(\frac{\sigma}{t}x\right)dx\\&\le
\frac{\pi|h'|_\infty|f'|^2_{\infty}}{4r}\frac{1}{t^2}\int_0^{t}\omega(\pi|f'|_{\infty}x)dx.\end{split}\end{equation}

By \eqref{eee}, having in mind the equation\eqref{goodd}, for
$r>1/2$ we obtain
\[\begin{split}\int_{-\pi}^{\pi} &\left|\mathcal{K}_F(e^{it},e^{i\tau})\frac
{P(r,\tau-t)}{1-r}\right|dt\le
\frac{2|h'|_\infty|f'|^2_{\infty}}{4r}\int_0^{\pi}\frac{1}{t^2}\int_0^{t}\omega(\pi|f'|_{\infty}u)du
\\&=\frac{|h'|_\infty|f'|^2_{\infty}}{2r}\int_0^{\pi}\left(\frac{\omega(\pi|f'|_{\infty}u)}{u}-\frac{\omega(\pi |F'|_{\infty}u)}{\pi}
\right)du \\&<M<\infty.\end{split}\]

 According to
Lebesgue Dominated Convergence Theorem, taking the limit under the
integral sign in the last integral in
 \eqref{last} we obtain \eqref{jacfor}.

\end{proof}

\begin{lemma}\label{jacabove}
Let $u=P[F](z)$ be a Lipschitz continuous  harmonic function between
the unit disk $\mathbf U$ and a $\,\mathscr C^{1,\mu}$ surface
$\mathcal M\subset \mathbf R^n$ such that $F$ is injective, and
$\partial D=h(\mathbf T)\in \,\mathscr C^{1,\mu}$, where $h$ is a
$\,\mathscr C^{1,\mu}$ parametrization. Then for almost every
$e^{i\varphi} \in \mathbf T$ we have
\begin{equation}\label{jakobo}\limsup_{r \to 1-0} |J_u(re^{i\varphi})| \leq
C_h|f'(\varphi)|\int_{-\pi}^\pi
\dfrac{|F(e^{i(\varphi+x)})-F(e^{i\varphi})|^{1+\mu}}{\pi|e^{ix}-1|^2}
dx,\end{equation} where $J_u(z)$ denotes the Jacobian of $u$ at $z$,
 $F(e^{it})=h(e^{if(t)})$, and $$C_h= \frac{1}{(1+\mu){\min_{t}|h'(t)|}}\sup_{x\neq y}\frac{|h'(x)-h'(y)|}{|x-y|^\mu}.$$

\end{lemma}
\begin{proof}
It follows from Lemma~\ref{11} and Lemma~\ref{22}.
\end{proof}

\section{The main results}
Let $\gamma\subset \mathbf R^n$ be a closed rectifiable Jordan
curve. Let $d_\gamma(a,b)$ be the length of the shorter Jordan arc
of $\gamma$ with endpoints $a,b\in \gamma$. We say that $\gamma$
enjoys a $\lambda-$ chord-arc condition for some constant $\lambda>
1$ (or $\gamma$ is chord-arc) if for all $x,y\in \gamma$ there holds
the inequality
\begin{equation}\label{24march}
d_\gamma(x,y)\le \lambda|x-y|.
\end{equation}
It is clear that if $\gamma\in \,\mathscr C^{1}$ then $\gamma$
enjoys a chord-arc condition for some $\lambda_\gamma>1$.

\begin{lemma}\label{newle}
Assume that $\gamma\subset\mathbf R^n$ enjoys a chord-arc condition
for some $\lambda>1$ and that bounds a smooth
$\Upsilon$-isoperimetric surface $\mathcal M$ ($\partial \mathcal M
= \gamma$). Then for every $K-$ q.c. normalized mapping $u$ between
the unit disk $\mathbf U$ and $\mathcal M$ there holds the
inequality \begin{equation}|u(z_1)-u(z_2)|\le
L_\gamma(K)|z_1-z_2|^\alpha\end{equation} for $z_1,z_2\in \mathbf
T$, $\alpha = \frac{8\Upsilon}{\pi K(1+2\lambda)^2}$ and
$$L_\gamma(K) =4 (1+2\lambda)2^{\alpha}\sqrt{\frac{2\pi K|\mathcal
M|}{\log 2}},$$ where $|\mathcal M|$ is the area of the surface
$\mathcal M$.
\end{lemma}

\begin{proof}

For $a\in \mathbf C$ and $r>0$, put $D(a,r):=\{z:|z-a|<r\}$.
 It is clear that if $z_0\in \mathbf T=\partial \mathbf U$, then, because of normalization, $u(\mathbf T\cap
\overline{D(z_0,1)})$ has common points with at most two of three
arcs $\wideparen{w_0w_1}$, $\wideparen{w_1w_2}$ and
$\wideparen{w_2w_0}$. (Here $w_0$, $w_1$, $w_2\in \gamma$ divide
$\gamma$ into three arcs with the same length such that $u(1)=w_0$,
$u(e^{2\pi i/3})=w_1$, $u(e^{4\pi i/3})=w_2$, and $\mathbf T\cap
\overline{D(z_0,1)}$ do not intersect at least one of three arcs
defined by $1$, $e^{2\pi i/3}$ and $e^{4\pi i/3}$).

Let $\varphi$ be a conformal mapping of a some neighborhood of
$\mathcal M$ onto the unit disk (this neighborhood exist according
to the definition od disk-type surfaces with boundary). Then
$u_1=\varphi \circ u$ is a $K$ quasiconformal mapping of the unit
disk onto a domain with rectifiable boundary $\gamma_1$ which enjoys
a chord-arc condition. According to \cite[Theorem~5, p. 81]{Ahl}, it
admits a q.c. reflection. Let $\tilde u$ be a q.c. mapping of the
whole plane onto itself, such that $\tilde u|_{\mathbf U} = u_1$.

Let $k_\rho$ denotes the arc of the circle $|z-z_0|=\rho<1$ which
lies in $|z|\le 1$. Take $ F(\rho,\varphi)=\tilde u_1(z_0 -
\overline{z_0}\rho e^{i\varphi})$. Then for $n\in \mathbf N$ $F$ is
quasiconformal in $A_n=[\frac{1}{n}, n]\times[-\pi/2-1,\pi/2+1]$.
Since $k_\rho\subset \{z_0 - \overline{z_0}\rho
e^{i\varphi}:\varphi\in (-\pi/2,\pi/2)\}$, and $F$ is absolutely
continuous on almost every line $\{\rho\}\times[-\pi/2,\pi/2]$, it
follows that $u_1$ is absolutely continuous on almost every $k_\rho$
for $\rho>\frac 1n$ as well as the mapping $u$. The conclusion is
that $u$ is absolutely continuous on almost every $k_\rho$ for
$\rho>0$.

Let $l_\rho=|u(k_\rho)|$ denotes the length of $u(k_\rho)$. Let
$\kappa_\rho=\{t\in[0,2\pi]: z_0+\rho e^{it}\in k_\rho\}$. By using
 polar coordinates and the Cauchy-Schwarz inequality, we have for
almost every $\rho$
\[\begin{split}l_\rho^2&=|u(k_\rho)|^2=\left(\int_{k_\rho}|du|\right)^2\\&\le \left(\int_{\kappa_\rho}|\nabla u(z_0+\rho
e^{i\varphi})| \rho d\varphi\right)^2 \\&\le
\int_{\kappa_\rho}|\nabla u(z_0+\rho e^{i\varphi})|^2 \rho
d\varphi\cdot \int_{\kappa_\rho} \rho d\varphi.\end{split}\]

Let $\gamma_\rho:= u(\mathbf T\cap D(z_0,\rho))$ and let
$|\gamma_\rho|$ be its length. Assume $w$ and $w'$ are the endpoints
of $\gamma_\rho$, i.e. of $u(k_\rho)$. Then $|\gamma_\rho| =
d_\gamma(w,w')$ or $|\gamma_\rho| = |\gamma| - d_\gamma(w,w')$. If
the first case holds, then since $\gamma$ enjoys the
$\lambda-$chord-arc condition, it follows $|\gamma_\rho|\le
\lambda|w-w'|\le \lambda l_\rho$. Consider now the last case. Let
$\gamma_\rho' = \gamma\setminus \gamma_\rho$. Then $\gamma_\rho'$
contains one of the arcs $\wideparen{w_0w_1}$, $\wideparen{w_1w_2}$,
$\wideparen{w_2w_0}$. Thus $|\gamma_\rho|\le 2|\gamma_\rho'|$, and
therefore
$$|\gamma_\rho|\le 2\lambda l_\rho.$$
Since $l(k_\rho)\le 2\rho \pi/2$, for $r\le 1$, denoting
$\Delta_r=\mathbf U\cap D(z_0,r)$, we have

\begin{equation}\label{hen}\begin{split}\int_0^r\frac{l^2_\rho}{\rho}d\rho &\le
\pi\int_0^r\int_{\kappa_\rho} |\nabla u(z_0+\rho e^{i\varphi})|^2
\rho d\varphi d\rho\\& \le \pi K \int_0^r\int_{\kappa_\rho}
J_u(z_0+\rho e^{i\varphi}) \rho d\varphi d\rho = \pi
A(r)K,\end{split}\end{equation} where $A(r)$ is the area of
$u(\Delta_r)$. Using the first part of the proof, it follows that
the length of boundary arc $\gamma_r$ of $u(\Delta_r)$ does not
exceed $2\lambda l_r$ which, according to the fact that $\partial
u(\Delta_r)=\gamma_r \cup u(k_r)$, implies
\begin{equation}\label{nevoja}|\partial u(\Delta_r)| \le l_r+2\lambda l_r.\end{equation} Therefore, by
the isoperimetric inequality $$ A(r) \le \frac{|\partial
u(\Delta_r)|^2}{4\Upsilon}\le \frac{(l_r +2\lambda
l_r)^2}{4\Upsilon} = l^2_r\frac{(1+2\lambda)^2}{4\Upsilon}.$$
Employing now \eqref{hen} we obtain
$$F(r):=\int_0^r\frac{l^2_\rho}{\rho}d\rho \le
\pi Kl^2_r\frac{(1+2\lambda)^2}{4\Upsilon}.$$ Observe that for
$0<r\le 1$ there holds $rF'(r) = l^2_r$. Thus $$F(r)\le\pi  K rF'(r)
\frac{(1+2\lambda)^2}{4\Upsilon}.$$ It follows that for $$\alpha =
\frac{8\Upsilon}{\pi K(1+2\lambda)^2}$$ there holds
$$\frac{d}{dr}\log (F(r)\cdot r^{-2\alpha})\ge 0,$$ i.e. the function
$F(r)\cdot r^{-2\alpha}$ is increasing. This yields

$$F(r)\le F(1)r^{2\alpha}\le \pi K |\mathcal M|
r^{2\alpha}.$$ Now for every $r\le 1$ there exists an
$r_1\in[r/\sqrt 2,r]$ such that
$$F(r)=\int_0^r\frac{l^2_\rho}{\rho}d\rho\ge \int_{r/\sqrt
2}^r\frac{l^2_\rho}{\rho}d\rho = l^2_{r_1}\log \sqrt 2. $$ Hence
$$l^2_{r_1}\le \frac{2\pi K|\mathcal M|}{\log 2} r^{2\alpha}.$$ If
$z$ is a point with  $|z|\le 1$ and $|z-z_0|=r/\sqrt 2$, then by
\eqref{nevoja}
$$|u(z)-u(z_0)|\le (1+2\lambda)l_{r_1}.$$ Therefore $$
|u(z)-u(z_0)|\le H|z-z_0|^\alpha, $$ where $$H =
(1+2\lambda)2^{\alpha/2}\sqrt{\frac{2\pi K|\mathcal M|}{\log 2}}.$$

Thus we have for $z_1,z_2\in \mathbf T$ the inequality
\begin{equation}\label{there}|u(z_1)-u(z_2)|\le
4H|z_1-z_2|^\alpha.\end{equation}
\end{proof}

We now prove the main result of this paper.
\begin{theorem}\label{theom}
If $u: \mathbf U\to \mathcal M$ is a normalized $K-$quasiconformal
harmonic mapping of the unit disk onto a  $\,\mathscr C^{1,\mu}$
surface $\mathcal M$ with $\,\mathscr C^{1,\mu}$ Jordan boundary
$\gamma$, then $u$ is Lipschitz and there exists a constant
$L=L(K,\gamma)$ (satisfying inequality~\eqref{L} below) such that
$$|\nabla u|\le L$$ and
$$|u(z)-u(w)|\le KL|z-w|.$$
\end{theorem}

\begin{proof}  Let $\tau$ be a conformal
mapping of the unit disk onto $\mathcal M$ described in
Remark~\ref{rema}. Then $\tau$ is bi-Lipschitz, i.e. there exists a
constant $c_{\mathcal M}\ge 1$ such that $$\frac 1{c_{\mathcal
M}}\le \frac{|\tau(z)-\tau(w)|}{|z-w|}\le c_{\mathcal M}, \ z\neq w,
z,w\in\mathbf U.$$ Let in addition $\varphi_r$ be a conformal
mapping of the unit disk onto $U_r = u^{-1}(\tau(r\mathrm U))$ such
that the mapping $u_r= u\circ \varphi_r=P[F_r]$ is normalized. The
mapping $u_r$ is a $K$ quasiconformal harmonic mapping of the unit
disk onto the surface $\mathcal M_r\subset \mathcal M$ with boundary
$\gamma_r=\partial\mathcal M_r$, satisfying $\lambda'$ -chord-arc
condition, where \begin{equation}\label{lpr}\lambda' = \frac{ \pi
c_{\mathcal M}^2}{2}.\end{equation}
 Namely let $a=\tau(re^{it})$ and $b=\tau(re^{is})$ be
two points of the Jordan curve $\gamma_r$ and let
$d_{\gamma_r}(a,b)$ be the length of shorter Jordan arc of
$\gamma_r$ with endpoints $a$ and $b$. Then
$$d_{\gamma_r}(a,b)\le\int_{\wideparen{[re^{it},re^{is}]}}|\tau'(z)||dz|\le \frac{\pi c_{\mathcal M}}{2}|re^{is}-re^{it}|\le
\frac{\pi c_{\mathcal M}^2}{2}|a-b|.$$ Here
$\wideparen{[re^{it},re^{is}]}$ is the shorter act of $r\mathbf T$
with endpoints $re^{it}$ and $re^{is}$.

Consider the parametrization $\tau_r:\mathbf T \to \gamma_r$ defined
by $\tau_r(e^{it})=\tau(re^{it})$, and write $\tau_r(t)$ instead of
$\tau_r(e^{it})$. First of all
$$\tau_r'(t) = r \tau'(re^{it})$$ and there exists a constant $c_\tau$ such that \begin{equation}\label{ctau}\frac{r}{c_\tau}\le |\tau'_r(t)|\le
rc_\tau.\end{equation} On the other hand,
$$C_{\tau_r}:=\sup_{t\neq s}\frac{|\tau_r'(s)-\tau_r'(t)|}{|t-s|^\mu}=\sup_{t\neq s}\frac{r|\tau'(re^{is})-\tau'(re^{it})|}{|t-s|^\mu}\le
r C_{\mathcal M},$$ where $C_{\mathcal M}$ is defined in \eqref{cm}.
\\
Let $$L_r = |F_r'|_{\infty}:=\max\{|F'_r(e^{is})|:
s\in[0,2\pi]\}=|F'_r(e^{it})|.$$ First of all
\begin{equation}\label{first}
|F_r(e^{i(t+x)})-F_r(e^{it})|\le \frac{\pi}{2}L_r|e^{ix}-1|.
\end{equation}

By \eqref{december1}, (\ref{jakobo}) (taking $h:=\tau_r$) and
\eqref{ctau} we obtain
$$|F_r'(t)|^2\le \frac{C_{\tau_r}K}{2rc_\tau}
|F_r'(t)|\int_{-\pi}^\pi
\dfrac{|F_r(e^{i(t+x)})-F_r(e^{it})|^{1+\mu}}{|e^{ix}-1|^2} dx,$$
i.e. $$|F_r'(t)|\le \frac{C_{\tau_r}K}{2rc_\tau} \int_{-\pi}^\pi
\dfrac{|F_r(e^{i(t+x)})-F_r(e^{it})|^{1+\mu}}{|e^{ix}-1|^2} dx.$$
Hence for \begin{equation}\label{televizor}C_r=
\frac{C_{\tau_r}K}{2rc_\tau},\end{equation} and for $\beta$
satisfying $0<\beta<1$, in view of \eqref{first}, we obtain
\begin{equation}\label{mda}\begin{split}L_r&\le C_r\int_{-\pi}^\pi
\dfrac{|F_r(e^{i(t+x)})-F_r(e^{it})|^{1+\mu}}{|e^{ix}-1|^{1-\mu}}
\frac{dx}{|e^{ix}-1|^{1+\mu}}\\&\le C_r\int_{-\pi}^\pi
\dfrac{|F_r(e^{i(t+x)})-F_r(e^{it})|^{1+\mu-\beta}}{|e^{ix}-1|^{1+\mu-\beta}}L^\beta
\frac{dx}{|e^{ix}-1|^{1-\mu}}.\end{split}\end{equation} Thus
$$L_r/L_r^\beta\le C_r\int_{-\pi}^\pi
\dfrac{|F_r(e^{i(t+x)})-F_r(e^{it})|^{1+\mu-\beta}}{|e^{ix}-1|^{1+\mu-\beta}}
\frac{dx}{|e^{ix}-1|^{1-\mu}}.$$ In view of Definition~\ref{perdef},
harmonic surfaces are $1$-isoperimetric. By Lemma~\ref{newle} $F_r$
is H\"older continuous with exponent
\begin{equation}\label{apr}{\alpha'} = \frac{2\Upsilon}{\pi
K(1+2\lambda')^2}\ge\frac{2}{\pi K(1+2\lambda')^2}.\end{equation}
Choose $\beta$, $0<\beta<1$, sufficiently close to 1, so that
$$\sigma=({\alpha'}-1)(1+\mu-\beta)+\mu-1>-1.$$ For example,
$$\beta = 1-\frac{\mu{\alpha'}}{2-{\alpha'}},$$ and consequently, $$\sigma=
\frac{\mu{\alpha'}}{2-{\alpha'}}-1.$$ Because $\gamma_r$ is a
$\lambda'$ chord-arc curve, we get
$$L_r^{1-\beta}\le C_r\cdot ( L_{\gamma_r}(K))^{1+\mu-\beta}
\int_{-\pi}^\pi |e^{ix}-1|^{\sigma}dx=C_r',$$ and hence
\begin{equation}\label{eqmc}L_r\le
(C_r')^{1/(1-\beta)}=(C_r')^{\frac{2-{\alpha'}}{\mu{\alpha'}}}.\end{equation}
%
On the other hand \[\begin{split}C_r'&\le C_r
(L_\gamma(K))^{1+\mu-\beta}\frac{2^{1 + \sigma} \pi^{3/2} \sec(\pi
\sigma/2)}{ \Gamma(1/2 - s/2) \Gamma(1 + s/2)}\\&\le
\frac{C_{\mathcal M}K}{2c_\tau
}(L_\gamma(K))^{1+\mu-\beta}\frac{2^{1 + \sigma} \pi^{3/2} \sec(\pi
\sigma/2)}{ \Gamma(1/2 - \sigma/2) \Gamma(1 +
\sigma/2)}\\&\le\frac{C_{\mathcal M}K}{c_\tau }\frac{2^{
\sigma}\pi}{1+\sigma}(L_\gamma(K))^{1+\mu-\beta} .\end{split}\]
Thus, in view of \eqref{lpr} and \eqref{apr} we obtain
\begin{equation}\label{lin}\sup_r L_r<\infty.
\end{equation}
Since $F_r$ is smooth we have
$$\frac{\partial u_r}{\partial t}(\rho e^{it})=P[F_r'](\rho e^{it}).$$  From
\eqref{1} and \eqref{igi} it follows
\begin{equation}\label{var1a}|\nabla u_r(e^{it})|\le K\left|\frac{\partial u_r}{\partial t}(e^{it})\right|\le K |F'_r|_{\infty}.\end{equation}

To continue, observe that $\nabla u_r$ is harmonic and bounded and
there holds
$$\nabla u_r(z)=P[\nabla u_r(e^{it})](z).$$ Thus for $|h|=1$ $$|\nabla u_r(z) h|=|P[\nabla u_r(e^{it})h](z)|\le P[|\nabla u_r(e^{it})h|](z).$$
It follows that
\begin{equation}\label{vara}|\nabla u_r|\le  K |F'_r|_{\infty}=KL_r.\end{equation}
Combining \eqref{vara} and \eqref{lin} we obtain
$$\sup\{|\nabla u(z)|: {|z|< 1}\}<\infty.$$ By using the fact proved in the previous proof that $u$ is Lipschitz continuous, and proceeding again the previous
proof: setting $L$ instead of $L_r$, $f$ instead of $f_r$, and an
arc length parametrization $g$ of $\gamma$ instead of $\tau_r$, and
using the fact that a smooth curve is chord-arc curve for some
$\lambda$, we obtain that
$$\mathrm{ess\,sup}\{|F'(t)|: 0\le t\le 2\pi\}=L(K,\lambda,\Upsilon)=:L,$$
and
\begin{equation}\label{varal}|\nabla u|\le  K |F'|_{\infty}=KL.
\end{equation} See remark below for an explicit estimate of $L$.
\end{proof}
\begin{remark}\label{rem}
Since $4|\mathcal M|\le |\gamma|^2$, the previous proof yields the
following estimate of a Lipschitz constant $L$ for a normalized
$K-$quasiconformal harmonic mapping between the unit disk and a
disk-type surface $\mathcal M$ bounded by a Jordan curve $\gamma\in
\,\mathscr C^{1,\mu}$ satisfying a $\lambda-$chord-arc condition.
\begin{equation}\label{L}L\le 8\left(K
C_\gamma\frac{\pi(2-\alpha)}{2\mu\alpha}\right)^{\frac{2-\alpha}{\mu\alpha}}
\left\{4 (1+2\lambda)|\gamma|\sqrt{\frac{\pi K}{\log
4}}\right\}^{\frac{2}{\alpha}},\end{equation} where $$\alpha =
\frac{8\Upsilon}{\pi K(1+2\lambda)^{2}}, \ \Upsilon\ge 1$$ and
$$C_{\gamma}=\max_{s\neq t}\frac{|g'(t)-g'(s)|}{|t-s|^\mu}.$$ Here $g$ is an
arc length parametrization of $\gamma$. See \cite{MP}, \cite{ps},
\cite{MMM} and \cite{trans} for more explicit (more precise)
constants, in the special case where $\gamma$ is the unit circle. We
can express $\lambda$ in terms of $g$ as follows
\begin{equation}\label{lama}\lambda = \sup_{s\neq t}\frac{|s-t|}{|g(s)-g(t)|}.\end{equation} It
is clear that $\lambda<\infty$, because $$\lim_{(s,t)\to
(\tau,\tau)}\frac{|s-t|}{|g(s)-g(t)|}=\frac{1}{|g'(\tau)|}=1.$$
\end{remark}
The following corollary of Theorem~\ref{theom} gives a quantitative
estimate of the lipschitz constant of a conformal mapping,
parameterizing a minimal surface with smooth boundary (see
\cite{nit} and \cite{les} for existential proof of this fact and
\cite{SW} and \cite{SW1} for related results).
\begin{corollary}\label{coco}
Let $u(x,y)=(u^1,\dots,u^n):\mathbf U\to \mathcal M$ be normalized
isotherm coordinates of a minimal surface $\mathcal M$ spanning a
Jordan curve $\gamma\in \,\mathscr C^{1,\mu}$ and assume that $g$ is
an arc-length parametrization of $\gamma$. Then $$|u_x|=|u_y|\le
L,$$ where
$$L= 8\left\{\frac{ C_\gamma(-3/4+
\lambda(1+\lambda))\pi}{2\mu}\right\}^{\frac{-3/4+\lambda(1+\lambda)}{\mu}}
\left\{\frac{4 (1+2\lambda)|\gamma|}{\sqrt{{\log
4}}}\right\}^{(1/2+\lambda)^2},$$ and $\lambda$ is defined in
\eqref{lama}. In particular if $\mu = 1$, we obtain
$$L= 8\left\{ \frac{\kappa_\gamma\pi}{2}\left[-\frac 34+ \lambda(1+\lambda)\right]\right\}^{-\tfrac 34+\lambda(1+\lambda)}
\left\{\frac{4 (1+2\lambda)|\gamma|}{\sqrt{{\log
4}}}\right\}^{(1/2+\lambda)^2},
$$ where $\kappa_\gamma$ is the largest curvature of $\gamma$
defined as $$\kappa_\gamma=\mathrm{ess}\sup_{0\le s\le
|\gamma|}|g''(s)|.$$
\end{corollary}
\begin{proof} Note first that, a minimal surface is not necessarily
regular. Thus we cannot apply directly Theorem~\ref{theom}. By
\cite[Theorem~1]{nit} (see also \cite[Theorem~1]{Lw}), $u$ is
Lipschitz continuous. Therefore the proof of Theorem~\ref{theom}
gives desired estimates. The constant is better than that of
\eqref{L}, because for minimal surfaces we have $\Upsilon = {\pi}$,
$K=1$ and $4\pi |\mathcal M|\le |\gamma|^2$.
\end{proof}

By using Theorem~\ref{theom} and isotherm coordinates near the
points in the boundary we obtain
\begin{theorem}
If $w: \mathcal N\to \mathcal M$ is a $K-$quasiconformal harmonic
mapping between $\,\mathscr C^{1,\mu}$ surfaces $\mathcal N$ and
$\mathcal M$ with $\,\mathscr C^{1,\mu}$ and compact boundaries,
then $w$ is Lipschitz.
\end{theorem}

\begin{proof}
Under the conditions of theorem, $w$ has a continuous extension to
the boundary.  Take $p\in \partial N$,  let $q=w(p)\in
\partial M$ and choose a disk-type surface $M_0$ with smooth
boundary, containing an open Jordan arc $\gamma_q\subset
\partial M$, such that $q\in \gamma_q$. Let $\varphi$ be a
conformal mapping of the unit disk onto $N_0=u^{-1}(\mathcal M_0)$,
such that $\varphi(1)=p$. Then $u = w\circ \varphi$ is a
$K-$quasiconformal harmonic mapping of the unit disk onto $\mathcal
M_0\subset \mathcal M$. By Theorem~\ref{theom} $u$ is Lipschitz.
According to Remark~\ref{rema} $\varphi$ is bi-Lipschitz in some
small neighborhood of $1$. This implies that $w$ is Lipschitz in
some neighborhood of $p$. As $\partial \mathcal N$ is compact, it
follows that $w$ is Lipschitz near the boundary $\partial \mathcal
N$ of $\mathcal N$. The conclusion is that, $w$ is Lipschitz.
\end{proof}

\subsection{Remarks}
Our main theorem has an important assumption that the surface
$\mathcal M$ is a $\,\mathscr C^{1,\mu}$ regular surface. Since the
Lipschitz constant depends only on the boundary of $\mathcal M$, it
would be interesting to check whether the assumption that $\mathcal
M$ is regular is essential. In the case that $\mathcal M$ is a
minimizing surface spanning a sufficiently smooth Jordan curve, then
$\mathcal M$ is a regular surface, i.e. $\mathcal M$ has not branch
points inside neither on the boundary (see for example \cite{bw} and
reference therein for this topic). We expect that, our main result
still hold assuming only that the corresponding mapping $u=P[F]$ is
a q.c. harmonic mapping such that $F(\mathbf T)$ is a $ \,\mathscr
C^{1,\mu}$ Jordan curve. On the other hand, harmonic surfaces can
have branch points as well as q.c. harmonic surfaces because minimal
surfaces that are not minimizing surfaces are not free of branch
points. Therefore, probably, the corresponding harmonic q.c. mapping
is not bi-Lipschitz, without the assumption that the surface is
regular. In a recent result of the author (\cite{surfaces}), it is
proved that, if the surface is at least $\,\mathscr C^{2,\mu}$
regular, then such mapping is bi-Lipschitz, and we expect that the
class $\,\mathscr C^{2,\mu}$ can be replaced by $\,\mathscr
C^{1,\mu}$.

\end{document}